\documentclass[11pt,reqno]{amsart}
\usepackage{amscd,verbatim,amssymb}
\usepackage{color}
\usepackage{verbatim}
\usepackage{hyperref}
\usepackage{xy}
\xyoption{all}
\usepackage{appendix}
\newtheorem{theorem}{Theorem}[section]
\newtheorem{proposition}[theorem]{Proposition}
\newtheorem{lemma}[theorem]{Lemma}
\newtheorem{corollary}[theorem]{Corollary}

\theoremstyle{definition}
\newtheorem{remark}[theorem]{Remark}
\newtheorem{example}[theorem]{Example}
\newtheorem{dfn}[theorem]{Definition}

\renewcommand{\(}{\bigl(}
\renewcommand{\)}{\bigr)}

\newcommand{\xra}[1]{\overset{#1}{\longrightarrow}}

\newcommand{\llg}{\longrightarrow}
\newcommand{\tens}{\otimes}

\newcommand{\tors}{\mathrm{tors}}

\newcommand{\CH}{\operatorname{CH}}

\renewcommand{\Im}{\operatorname{Im}}

\newcommand{\Ker}{\operatorname{Ker}}

\newcommand{\res}{\operatorname{res}}
\newcommand{\cor}{\operatorname{cor}}
\newcommand{\Br}{\operatorname{Br}}

\newcommand{\SB}{\operatorname{SB}}

\newcommand{\C}{\mathbb{C}}

\newcommand{\Z}{\mathbb{Z}}

\newcommand{\Q}{\mathbb{Q}}
\newcommand{\QZ}{\mathop{\mathbb{Q}/\mathbb{Z}}}

\numberwithin{equation}{section}

\begin{document}
\title[Decompososable and Indecomposable Algebras]{decomposable and Indecomposable Algebras of Degree $8$ and Exponent $2$}
\author[D. Barry]{Demba Barry
\\
With an appendix by Alexander S. Merkurjev}

\address{ICTEAM Institute, Universit\'e catholique de Louvain, B-1348 Louvain-la-Neuve, Belgium}
\email{demba.barry@uclouvain.be}
\address{Department of Mathematics, University of California, Los Angeles, CA, USA}
\email{merkurev@math.ucla.edu}

\date{}
\begin{abstract}
We study the decomposition of central simple algebras of exponent $2$ into tensor products of quaternion algebras. We consider in particular decompositions in which one of the quaternion algebras contains a given quadratic extension. Let $B$ be a biquaternion algebra over $F(\sqrt{a})$ with trivial corestriction. A degree $3$ cohomological invariant is defined and we show that it determines whether $B$ has a descent to $F$.  This invariant is used to give examples of indecomposable algebras of degree $8$ and exponent $2$ over a field  of $2$-cohomological dimension $3$ and over a field $\mathbb M(t)$ where the $u$-invariant of $\mathbb M$ is $8$ and $t$ is an indeterminate. The construction of these indecomposable algebras uses  Chow group computations provided by A. S. Merkurjev in Appendix.
\end{abstract}
\maketitle
\section{Introduction}
Let $A$ be a central simple algebra over a field $F$. We say that $A$ is \emph{decomposable} if $A\simeq A_1\otimes_F A_2$ for two central simple $F$-algebras $A_1$ and $A_2$ both non isomorphic to $F$; otherwise $A$ is called \emph{indecomposable}. Let $K=F(\sqrt{a})$ be a quadratic separable extension of $F$. We say that $A$ admits a \emph{decomposition adapted} to $K$ if $K$ is in a quaternion subalgebra of $A$, that is, $A\simeq (a,a')\otimes_F A'$ for some $a'\in F$ and some subalgebra $A'\subset A$. If $A$ is isomorphic to a tensor product of quaternion algebras, we will say that $A$  is  \emph{totally decomposable}. Let $B$ be a central simple algebra over $K$. The algebra $B$ has a \emph{descent} to $F$ if there exists an $F$-algebra $B'$ such that $B\simeq B'\otimes_F K$. It is clear that the algebra $A$ admits a decomposition adapted to $K$ if and only if the centralizer $C_AK$ of $K$ in $A$ has a descent to $F$.

The first example of a non-trivial indecomposable central simple algebra of exponent $2$ was given by Amitsur-Rowen-Tignol. More precisely,  an explicit central division algebra of degree $8$ and exponent $2$ which has no quaternion subalgebra is constructed in \cite{ART79}. Other examples of such algebras were given by Karpenko (see \cite{Kar95} and \cite{Kar98}) by computing torsion in Chow groups. In fact, $8$ is the smallest possible degree for such an algebra by a well-known theorem of Albert which asserts that every algebra of exponent $2$ and degree $4$ is decomposable.

The decomposability question depends on the cohomological dimension of the ground field. Indeed, for obvious reasons there is no indecomposable algebra of exponent $2$ over a field of cohomological dimension $0$ or $1$ (since the Brauer group is trivial in these cases). It follows from a result of Merkurjev (Theorem \ref{thm2.1}) that over a field of cohomological dimension $2$ any central simple algebra of exponent $2$ is isomorphic to a tensor product of quaternion algebras. On the other hand, the known examples of indecomposable algebras of exponent $2$ are constructed over fields of cohomological dimension greater than or equal to $5$.

The main goal of this article is to extend the existence of indecomposable algebras of exponent $2$ to some fields of cohomological dimension smaller than or equal to $4$. We also give an example over a field of rational functions in one variable over a field of $u$-invariant $8$. The problem will be addressed through the study of the decomposability adapted to a quadratic extension of the ground field. More precisely, let $K=F(\sqrt{a})\subset A$ be a quadratic extension field. We first prove  (Section \ref{sec2}) that if $cd_2(F)\le 2$ then $K$  lies in a quaternion subalgebra of $A$. If $\deg A=8$, we define a degree $3$ cohomological invariant, depending only on the centralizer $C_AK$ of $K$ in $A$, which determines whether $A$ admits a decomposition adapted to $K$, that is, whether $C_AK$ has a descent to $F$ (see Section \ref{sec3}). This invariant is used to give examples of indecomposable algebras of exponent $2$ (Theorem \ref{thmB} and Theorem \ref{thmC}). Although our invariant depends only on $C_KA$, it is nothing but a refinement of the invariant $\Delta(A)$ defined by Garibaldi-Parimala-Tignol \cite[\S 11]{GPT09} (see Remark \ref{rem2.1}). Through Remark \ref{rmq2.4} and the proof of Theorem \ref{thmB}, our invariant provides an example of indecomposable algebra $A$ of degree $8$ and exponent $2$ such that $\Delta(A)$ is nonzero, this is an answer to Garibaldi-Parimala-Tignol's question \cite[Question 11.2]{GPT09}.

For the proofs of Theorem \ref{thmB} and Theorem \ref{thmC} we need some results --- injections (\ref{eq2}) and (\ref{eq3}) --- on Chow groups of cycles of codimension $2$. These results follow by Theorem \ref{thmA.2} and Theorem \ref{thmA.1} provided by A. S. Merkurjev in Appendix.
\subsection*{Statement of main results}
Standard examples of fields of cohomological dimension $3$ are $k(t_1,t_2,t_3)$ or \, $k((t_1))((t_2))((t_3))$ where $k$ is an algebraically closed field and $t_1,t_2,t_3$ are independent indeterminates over $k$. Recall also that the $u$-invariant of such fields is $8$ by Tsen-Lang (see for instance \cite[Theorem 1]{Kah90}). The following theorem shows that there is no indecomposable algebra of degree $8$ and exponent $2$ over these standard fields. We recall that the \emph{u-invariant} of a field $F$ is defined as 
\[
u(F)=\max\{\dim \varphi\,|\, \varphi \text{ anisotropic form over } F\}.
\]
If no such maximum exists, $u(F)$ is defined to be $\infty$.
\begin{theorem}\label{thmA}
There exists no indecomposable algebra of degree $8$ and exponent $2$ over a field of $u$-invariant  smaller than or equal to $8$.
\end{theorem}
On the other hand, examples of indecomposable algebras do exist over a function field of one variable over a suitable field of $u$-invariant $8$:
\begin{theorem}\label{thmB}
Let $A$ be a central simple algebra of degree $8$ and exponent $2$ over a field $F$ and let $K=F(\sqrt{a})$ be a quadratic field extension of $F$ contained in $A$. If $K$ is not in a quaternion subalgebra of $A$ then there exists an extension $\mathbb M$ of $F$ with $u(\mathbb M)=8$  such that the division algebra Brauer equivalent to
\[
A_{\mathbb M}\otimes_{\mathbb M}(a,t)_{\mathbb M(t)}
\]
is an indecomposable algebra of degree $8$ and exponent $2$  over $\mathbb M(t)$, where $t$ is an indeterminate.
\end{theorem} 
To produce explicit examples, one may take for $A$ any indecomposable algebra, as those constructed in \cite{ART79} or \cite{Kar98}. Alternately, we give in \cite{Bar} an example of a decomposable algebra satisfying the hypothesis of Theorem \ref{thmB}. 

The following theorem shows that there exists an indecomposable algebra of degree $8$ and exponent $2$ over some field of $2$-cohomological dimension $3$. Let us recall that $\CH^2(\SB(A))_\tors$ is the torsion in the Chow group of cycles of codimension $2$ over the Severi-Brauer variety $\SB(A)$ of $A$  modulo rational equivalence. If $A$ is of prime exponent $p$ and index $p^n$ (except the case $p = 2 = n$) Karpenko showed in \cite[Proposition 5.3]{Kar98}  that if $\CH^2(\SB(A))_\tors\ne 0$ then $A$ is indecomposable. Examples of such  indecomposable algebras are given in \cite[Corollary 5.4]{Kar98}.
\begin{theorem}\label{thmC}
Let $A$ be a central simple algebra of degree $8$ and exponent $2$ such that $\CH^2(\SB(A))_\tors\ne 0$. Then there exists an extension $\mathbb M$ of $F$ with $cd_2(\mathbb{M})=3$ such that $A_{\mathbb M}$ is indecomposable.
\end{theorem}
\section{Notations and preliminaries}
Throughout this paper the characteristic of the base field $F$ is assumed to be different from $2$ and all algebras are associative and finite-dimensional over $F$. The main tools in this paper are central simple algebras, quadratic forms and Galois cohomology. Pierce's book \cite{Pie82} is a reference for the general theory of central simple algebras, references for quadratic form theory over fields are \cite{EKM08}, \cite{Lam05} and \cite{Sch85}. 

Let $F_s$ be a separable closure of $F$. For any integer $n\geq 0$, $H^n(F,\mu_2)$ denotes the Galois cohomology group
\[
H^n(F,\mu_2):= H^n(\text{Gal}(F_s/F),\mu_2)
\]
where $\mu_2=\{\pm 1\}$. The group $H^1(F, \mu_2)$ is identified with $F^\times/F^{\times 2}$ by Kummer theory. For any $a\in F^\times$, we write $(a)$ for the element of  $H^1(F, \mu_2)$ corresponding to $aF^{\times 2}$. The group $H^2(F, \mu_2)$  is identified with the $2$-torsion $\Br_2(F)$ in the Brauer group $\Br(F)$ of $F$, and we write $[A]$ for the element of $H^2(F, \mu_2)$ corresponding to a central simple algebra $A$ of exponent $2$. For more details on Galois cohomology the reader can consult Serre's book \cite{Ser97}.

For any quadratic form $q$, we denote by $C(q)$ and $d(q)$ the Clifford algebra and the discriminant of $q$. If $q$ has dimension $2m+2$, it is easy to see that $C(q)$ is a tensor product of $m+1$ quaternion algebras. Conversely, any tensor product of quaternion algebras is a Clifford algebra. Let $I^nF$ be the $n$-th power of the fundamental ideal $IF$ of the Witt ring $WF$. Abusing notations, we write $q\in I^nF$ if the Witt class of $q$ lies in $I^nF$. We shall use frequently the following property: if $q\in I^2F$ (i.e, $\dim q$ is even and its discriminant is trivial) the Clifford algebra of $q$ has the form $C(q)\simeq M_2(E(q))$ for some central simple algebra $E(q)$ which is totally  decomposable.

Now, let $E/F$ be an extension of  $F$ and let $q$ be a quadratic form defined over $F$. We say that $E/F$ is \emph{excellent}\index{excellent extension} for $q$ if the anisotropic part $(q_E)_{\text{an}}$ of $q_E$ is defined over $F$. If $E/F$ is excellent for every quadratic form defined over $F$, the extension $E/F$ is called \emph{excellent}. 

Generally, the biquadratic extensions are not excellent (see for instance \cite[\S 5]{ELTW83}) but we have the following result:

\begin{lemma}\label{lm4.1}
Assume that $I^3F=0$ and let $L=F(\sqrt{a_1},\sqrt{a_2})$ be a biquadratic extension of $F$. Then, $L/F$ is excellent for the quadratic forms $q\in I^2F$. More precisely, for all $q\in I^2F$ there exists $q_0\in I^2F$ such that $(q_L)_{\text{an}}\simeq (q_0)_L$.
\end{lemma}
\begin{proof}
Let $s: K=F(\sqrt{a_1})\to F$ be the $F$-linear map such that $s(1)=0$ and $s(\sqrt{a_1})=1$. The corresponding Scharlau's transfer  will be denoted by $s_*$. Notice that $I^3K=0$ because of $I^3F=0$ and the exactness of the sequence (see \cite[Theorem 40.3]{EKM08})
\[
\xymatrix { 
   \langle 1,-a_1\rangle I^2F\ar[r] & I^3F\ar[r] & I^3K\ar[r]^{s_*} &I^3F.
  }
\]
Let $q\in I^2F$ be an anisotropic form such that $q_L$ is isotropic. We first show that there exists a quadratic form $q_0$ defined over $F$ such that $(q_0)_L\simeq (q_L)_{\text{an}}$. If $q_L$ is hyperbolic, there is nothing to show. We suppose $q_L$ is not hyperbolic. It is well-known (see e.g. \cite[Theorem 3.1, p.197]{Lam05}) that  $(q_K)_{\text{an}}$ is defined over $F$. If $(q_K)_{\text{an}}$ remains anisotropic over $L$, we take $q_0=(q_K)_{\text{an}}$. Otherwise, one has
\[
(q_K)_{\text{an}}=\langle 1, -a_2\rangle\otimes\langle\gamma_1,\ldots, \gamma_r\rangle\perp q'
\]
for some $\gamma_1,\ldots, \gamma_r\in K$ and some subform $q'$ of $(q_K)_{\text{an}}$ defined over $K$ with $q'_L$ anisotropic (see for instance \cite[Proposition 3.2.1]{Kah08}). On the other hand, the form $(q_K)_{\text{an}}$ being in $I^2K$ and $I^3K=0$, one has $(q_K)_{\text{an}}\otimes \langle 1, -\gamma_1\rangle=0$, that is, $(q_K)_{\text{an}}\simeq \gamma_1 (q_K)_{\text{an}}$. Thus, we may suppose $\gamma_1=1$.

We claim that the dimension of the form $\langle 1,\gamma_2\ldots, \gamma_r\rangle$ is $1$.   Indeed, assume that $r\geq 2$ and write $(q_K)_{\text{an}}=\langle 1, -a_2\rangle\otimes\langle 1,\gamma_2\rangle\perp \varphi$ for some $\varphi$ over $K$. Since $\varphi_L$ is nonzero, it is clear that $\varphi$ is nonzero. Let $\alpha \in K^\times$ be represented by $\varphi$. The form $\langle 1,-a_2\rangle\otimes \langle 1, \gamma_2\rangle \otimes \langle 1,\alpha\rangle$  is hyperbolic since $I^3K=0$, hence $\langle 1,-a_2\rangle\otimes \langle 1, \gamma_2\rangle $ represents $-\alpha$. It follows that $(q_K)_{\text{an}}$ is isotropic, a contradiction. Hence $r=1$ and so $(q_K)_{\text{an}}\simeq \langle 1,-a_2\rangle\perp q'$.

Now, we are going to show that $q'_L\simeq (q_0)_L$ for some $q_0$ defined over $F$. The Scharlau transfers $s_*((q_K)_{\text{an}})$  and $s_*(\langle 1, -a_2\rangle)$ are both hyperbolic because $(q_K)_{\text{an}}$ and $\langle 1, -a_2\rangle$ are defined over $F$. This implies that $s_*(q')$ is hyperbolic, so $q'$ is Witt equivalent to $(q_0)_K$ for some $q_0$ defined over $F$. We deduce from  the excellence of $K/L$ that $q'\simeq (q_0)_K$.  Whence $(q_L)_{\text{an}}\simeq (q_0)_L$.  
   
It remains to prove that $q_0$ may be chosen to be in $I^2F$. Since $(q_0)_L\in I^2L$, the discriminant $d(q_0)\in \{F^{\times 2}, \, a_1.F^{\times 2},\, a_2.F^{\times 2},\, a_1a_2.F^{\times 2}\}$. If $d(q_0)=1$, we are done. Otherwise, assume for example $d(q_0)=a_1$ and write $q_0\simeq\langle c_1,\ldots, c_m\rangle$. The quadratic form $q_1\simeq\langle a_1c_1,c_2,\ldots, c_m\rangle$ is such that $(q_1)_L\simeq (q_0)_L$ and $d(q_1)=1$; that is $q_1\in I^2F$. It suffices to replace $q_0$ by $q_1$. This concludes the proof.
\end{proof}
\section{Adapted decomposition under $cd_2(F)\le2$} \label{sec2}
In this section we assume that the $2$-cohomological dimension $cd_2(F)\le 2$. If the characteristic of $F$ is different from $2$, Merkurjev proved that division algebras of exponent $2$ over $F$ are totally  decomposable. We use this observation to show that any  quadratic field extension of $F$ in a central simple algebra over $F$ of exponent $2$ lies in a quaternion subalgebra. We also prove a related result for a biquadratic extension of $F$. First, let us recall the following result:
\begin{theorem}[Merkurjev]\label{thm2.1} 
Assume that $cd_2(F)\le 2$. Then
\begin{itemize}
\item[(1)] Any central division algebra over $F$ whose class is in $\Br_2(F)$ is totally decomposable.
\item[(2)] A quadratic form $q\in I^2F$ is anisotropic if and only if $E(q)$ is a division algebra.
\end{itemize}
\end{theorem}
\begin{proof}
See \cite[Theorem 3]{Kah90}.
\end{proof}

Let $A$ be a $2$-power dimensional central simple algebra over $F$ and let $K\subset A$ be a quadratic extension of $F$. If $A$ is not a division algebra, $K$ is in any split quaternion algebra. We have the following general result in the division case:
\begin{proposition}\label{prop2.1}
Assume that $cd_2(F)\le 2$. Let $A$ be a central division  algebra of exponent $2$ over $F$. Every square-central element of $A$ lies in a quaternion subalgebra.
\end{proposition}
\begin{proof}
Assume that $\deg A=2^m$. By Theorem \ref{thm2.1}, the algebra $A$ decomposes into a tensor product of $m$ quaternion algebras. We may associate with $M_2(A)$ a $(2m+2)$-dimensional quadratic form $q\in I^2F$ in such a way that $C(q)\simeq M_2(A)$. This form $q$ is anisotropic by Theorem \ref{thm2.1} (2). The form $q$ being in $I^2F$, the center of the Clifford algebra $C_0(q)$ is $F\times F$ and $C_0(q)\simeq C_+(q)\times C_-(q)$ with $A\simeq C_+(q)\simeq C_-(q)$. Let $x\in A-F$ be such that $x^2=a\in F^{\times}$. The element $z=x^{-1}\sqrt{a}$, which is different from $1$ and $-1$, is such that $z^{2}=1$; hence $A_{F(\sqrt{a})}$ is not a division algebra. Then Theorem \ref{thm2.1} and \cite[Theorem 3.1, p.197]{Lam05}  indicate that the form $q$ has a subform $\alpha\langle 1,-a \rangle$ for some $\alpha\in F^{\times}$. Write $q\simeq \alpha\langle 1,-a \rangle\perp q_1$ and let $V_q$ be the underlying space of $q$. Let $e_1,\ldots, e_{2m+2}$ be an orthogonal basis of $V_q$ which corresponds to this diagonalization. One has:
\[
\left\{ \begin{array}{ll}
e_1e_2, e_1e_3 \in C_0(q)\\
(e_1e_2)^2=\alpha^2a,\,\, (e_1e_3)^2=-\alpha^2q(e_3) \\
(e_1e_2)(e_1e_3)=-(e_1e_3)(e_1e_2).
\end{array} \right.
\]
Let $z=e_1\ldots e_{2m+2}$. Since $q\in I^2F$, we have $z^2\in F^{\times 2}$. Let $z^2=\lambda^2$ with $\lambda\in F^{\times}$ and set  $z_+=\frac{1}{2}(1+\lambda^{-1}z)$ and $z_-=\frac{1}{2}(1-\lambda^{-1}z)$. So, $z_+, z_-$ are the primitive central idempotents in $C_0(q)$. The elements $(e_1e_2)z_+$ and $(e_1e_3)z_+$ generate a quaternion subalgebra isomorphic to $(a, a')$ in $C_+(q)\simeq A$ where $a'=(e_1e_3)^2$. Hence, $A$ contains some element $y$ such that $y^2=a$ and $y$ lies in a quaternion subalgebra of $A$. By the Skolem-Noether Theorem $x$ and $y$ are conjugated. Therefore,  $x$ is in a quaternion subalgebra. 
\end{proof}

Recall that if $cd_2(F)\le 2$ then $I^3F=0$ (\cite[Theorem A5']{LLT93}). This fact allows to use Lemma \ref{lm4.1} in the proof of the following results.
\begin{theorem}\label{thm2.2}
Assume that $cd_2(F)\le 2$. Let $A$ be a central division algebra of degree $2^n$ and exponent $2$ over $F$. Let $L=F(\sqrt{a_1},\sqrt{a_2})$ be a biquadratic extension of $F$ contained in $A$. There exist quaternion algebras $Q_1$, $Q_2$ and a subalgebra $A'\subset A$ over $F$ such that $F(\sqrt{a_i})\subset Q_i$ $(i=1,2)$ and $A\simeq Q_1\otimes Q_2\otimes A'$.
\end{theorem}
\begin{proof}
Let $q\in I^2F$ be an anisotropic quadratic form such that $C(q)\simeq M_2(A)$ (Theorem \ref{thm2.1}). Notice that the index of $A_L$ is $2^{n-2}$. Since $cd_2(L)=2$, it follows by \cite[Lemma 8]{Siv05} that 
\[
\dim(q_L)_{\text{an}}=2\log_2\text{ind}(A_L)+2= 2n-2.
\]
By Lemma \ref{lm4.1}, there is $q'\in I^2F$ such that $(q_L)_{\text{an}}\simeq q'_L$. Put $\psi=q-q'$. The form $\psi_L$ being hyperbolic it follows by \cite[Theorem 4.3, p.444]{Lam05}  that there exist quadratic forms $\varphi_1, \varphi_2$ such that  
\[
\psi= \langle 1, -a_1\rangle\otimes\varphi_1 \perp \langle 1, -a_2\rangle\otimes\varphi_2 \,\, \text{ in }\, WF.
\]
Since $\psi\in I^2F$, we must have $\varphi_1, \varphi_2\in IF$. Taking the Witt-Clifford invariant of each side of the identity above, we obtain 
\[
A\otimes A' \sim (a_1,d(\varphi_1)) \otimes (a_2,d(\varphi_2)) 
\]
where $A'$ is such that $M_2(A')=C(q')$. This yields that
\[
A\simeq (a_1,d(\varphi_1)) \otimes (a_2,d(\varphi_2))\otimes A'.
\]
\end{proof}
If $\deg A=8$, the above theorem holds without the division condition on $A$ as shows the following:
\begin{proposition}\label{prop2.4}
Assume that $cd_2(F)\le 2$. Let $A$ be a central simple algebra of degree $8$ and exponent $2$ over $F$. Let $L=F(\sqrt{a_1},\sqrt{a_2})$ be a biquadratic extension of $F$ contained in $A$. There exist quaternion algebras $Q_1$, $Q_2$  and $Q$ over $F$ such that $F(\sqrt{a_i})\subset Q_i$ $(i=1,2)$ and $A\simeq Q_1\otimes Q_2\otimes Q$.
\end{proposition}
\begin{proof}
We are going to argue on the index ind$(A)$ of $A$: if ind$(A)\le 2$, the statement is clear since $L$ lies in any split algebra of degree $4$. If ind$(A)=8$, the result follows from Theorem \ref{thm2.2}. 

Suppose ind$(A)=4$. Let $q\in I^2F$ be an anisotropic quadratic form such that $C(q)\sim M_2(A)$. Such a form exists by Merkurjev's Theorem. Since $A_L$ is split or of index $2$, the form $q_L$ is either hyperbolic or equivalent to a multiple of the anisotropic norm form, $n_Q$, of some quaternion algebra $Q$ defined over $L$. Notice that every multiple of $n_Q$ is isometric to $n_Q$. We may consider $n_Q$ as being defined over $F$ because of Lemma \ref{lm4.1}; so $Q$ is defined over $F$. Moreover, $(q-n_Q)_L$ is hyperbolic. Put $\psi=q$ if ind$(A_L)=1$ and $\psi=q-n_Q$ if ind$(A_L)=2$; the form $\psi_L$ is hyperbolic. As in the proof of Theorem \ref{thm2.2} there exist quadratic forms $\varphi_1, \varphi_2\in IF$ such that  
\[
A\otimes Q \sim (a_1,d(\varphi_1)) \otimes (a_2,d(\varphi_2)) \,\, \text{ if }\,\, \text{ind}(A_L)=2,
\]
and 
\[
A\sim (a_1,d(\varphi_1)) \otimes (a_2,d(\varphi_2)) \,\, \text{ if }\,\, \text{ind}(A_L)=1.
\] 
This yields that
\[
A\simeq (a_1,d(\varphi_1)) \otimes (a_2,d(\varphi_2))\otimes M_2(F) \quad \text{or}\quad A\simeq (a_1,d(\varphi_1)) \otimes (a_2,d(\varphi_2))\otimes Q.
\]
\end{proof}
\begin{remark}\label{rem2.0}
For central simple algebras of degree $8$ and exponent $2$ over $F$ with $cd_2(F)=2$,  Proposition \ref{prop2.4} cannot be generalized to the triquadratic extensions of $F$. For example, let $F=\C(x,y)$ where $\C$ is the field of complex numbers and $x,y$ are independent indeterminates over $\C$. Notice that $cd_2(F)=2$. Since all $5$-dimensional quadratic forms over $F$ are isotropic by Tsen-Lang (see for instance \cite[p.376]{Lam05}), Theorem \ref{thm2.1} yields that index equals $2$ for every division algebra of exponent $2$ over $F$. Let $a_1=x^2(1+y)^2-4xy$, $a_2=1-x$, $a_3=y$ and $M=F(\sqrt{a_1},\sqrt{a_2},\sqrt{a_3})$. It is shown in \cite[(5.8)]{ELTW83} that there exists an algebra $A$ of exponent $2$ and degree $8$ over $F$ containing $M$ such that $A$ admits no decomposition of the form $(a_1,a_1')\otimes(a_2,a_2')\otimes(a_3,a_3')$ with $a_1',a_2', a_3'\in F$.
\end{remark}
\section{Degree 3 invariant}\label{sec3}
In this section, we assume no condition on the $2$-cohomological dimension of $F$. Let $K=F(\sqrt{a})$ be a quadratic field extension of $F$. Throughout this section $B$ is a biquaternion algebra over $K$ such that $\cor_{K/F}(B)$ is trivial. In fact, this condition on $\cor_{K/F}(B)$ means that $B$ has a descent to $F$ up to Brauer equivalence because of the exactness of the sequence (see for instance \cite[(4.6)]{Ara75})
\[
\xymatrix @C=3pc{ 
    \text{Br}_2(F) \ar[r]^{r_{K/F}} & \text{Br}_2(K) \ar[r]^{\text{cor}_{K/F}} & \text{Br}_2(F).
     }
\]
We are going to attach to $B$ a cohomological invariant which determines whether $B$ has a descent to $F$. This descent problem is another way of studying the question of adapted decomposition. Indeed, let $A$ be a central simple algebra over $F$ of degree greater than or equal to $8$ whose restriction is $B$. As explained in the introduction, $B$ has a descent to $F$ if and only if $A$ admits a decomposition adapted to $K$. 

Let $\varphi$ be an Albert form of $B$, that is, $\varphi$ is a $6$-dimensional form in $I^2K$ such that $C_0(\varphi)\simeq C_+(\varphi)\times C_-(\varphi)$ with $C_+(\varphi)\simeq C_-(\varphi)\simeq B$. Recall that an Albert form is unique up to a scalar. We have the following lemma  where $e_3$ denotes the Arason invariant $I^3F\to H^3(F,\mu_2)$: 

\begin{lemma}\label{lm4.2}
We keep the above notations. One has the following statements:

\begin{enumerate}
\item[(1)] Scharlau's transfer $s_*(\varphi)$ lies in $I^3F$.

\item[(2)] For any $\lambda\in K^{\times}$, we have $e_3(s_*(\varphi))=e_3(s_*(\langle\lambda\rangle\otimes\varphi))+ \cor_{K/F}((\lambda)\cdot[B])$.
\item[(3)] For $\lambda \in K^{\times}$, the following are equivalent:
\begin{itemize}
\item[(i)] $s_*(\langle\lambda\rangle\otimes\varphi)=0$.
\item[(ii)] $e_3(s_*(\varphi))=\cor_{K/F}((\lambda)\cdot[B])$.
\end{itemize}
\end{enumerate}
\end{lemma}
\begin{proof}
(1)  Since the diagram
\[
\xymatrix @R=3pc @C=3pc{ 
    I^2K \ar[r]^{s_*} \ar[d]_{e_2}  & I^2F \ar[d]^{e_2} \\
     \text{Br}_2(K) \ar[r]^{\cor_{K/F}} & \text{Br}_2(F)
     }
\]
commutes (see \cite[(4.18)]{Ara75}) and $\cor_{K/F}(B)=0$, we have $e_2(s_*(\varphi))=0$. Hence, 
\[
s_*(\varphi)\in\Ker(e_2\,:\, I^2F\rightarrow \text{Br}_2(F))=I^3F.
\]

\noindent
(2) From the identity
\[
s_*(\varphi)=s_*(\langle\lambda\rangle\otimes\varphi+\langle 1,-\lambda\rangle\otimes\varphi)=s_*(\langle \lambda \rangle\otimes\varphi)+ s_*(\langle 1,-\lambda\rangle\otimes\varphi)
\] 
we have
\begin{equation}\label{eq}
e_3(s_*(\varphi))= e_3(s_*(\langle\lambda\rangle\otimes\varphi))+e_3(s_*(\langle 1,-\lambda\rangle\otimes\varphi)).
\end{equation}
On the other hand, since the diagram
\[
\xymatrix @R=3pc @C=3pc{ 
    I^3K \ar[r]^{s_*} \ar[d]_{e_3}  & I^3F \ar[d]^{e_3} \\
     H^3(K,\mu_2) \ar[r]^{\cor_{K/F}} & H^3(F,\mu_2)
     }
\]
commutes (\cite[(5.7)]{Ara75}), one has
\begin{eqnarray*}
e_3(s_*(\langle 1,-\lambda\rangle\otimes\varphi)) & = & \cor_{K/F}(e_3(\langle 1,-\lambda\rangle\otimes\varphi)) \\    
 & = & \cor_{K/F}((\lambda)\cdot e_2(\varphi)) \\                      
 & = & \cor_{K/F}((\lambda)\cdot[C_+(\varphi)])\\                       
 & = & \cor_{K/F}((\lambda)\cdot [B]). 
\end{eqnarray*}
Therefore the relation (\ref{eq}) becomes
\[
e_3(s_*(\varphi))= e_3(s_*(\langle\lambda\rangle\otimes\varphi))+ \cor_{K/F}((\lambda)\cdot [B])
\]
as desired.

\noindent
(3) This point follows immediately from (2) because $s_*(\langle\lambda\rangle\otimes\varphi)=0$ if and only if $e_3(s_*(\langle\lambda\rangle\otimes\varphi)=0$ by the Arason-Pfister Hauptsatz (see for instance \cite[Chap. 4]{EKM08})
\end{proof}

It follows from Lemma \ref{lm4.2} (2) that $e_3(s_*(\varphi))$ modulo $\cor_{K/F}((K^\times)\cdot [B])$ does not depend on the choice of the Albert form $\varphi$. Therefore we may define an invariant: 
\begin{dfn}
Let $B$ be a biquaternion algebra over $K$ such that $\cor_{K/F}(B)=0$. The invariant
\[
\delta_{K/F}(B)\in \frac{H^3(F,\mu_2)}{\cor_{K/F}((K^\times)\cdot [B])}
\]
is the class of $e_3(s_*(\varphi))$.
\end{dfn}

\begin{remark}\label{rem2.1}
If $A$ is a central simple $F$-algebra of degree $8$ containing $K$, and if $B=C_AK$, then $\delta_{K/F}(B)$ is the image in $H^3(F,\mu_2)/\cor_{K/F}((K^\times)\cdot[B])$ of the discriminant $\Delta(A,\sigma)\in H^3(F,\mu_2)$ of any symplectic involution $\sigma$ on $A$ which leaves $K$ elementwise fixed. This follows by comparing the definition of $\delta_{K/F}(B)$ with Garibaldi-Parimala-Tignol \cite[Proposition 8.1]{GPT09}. Therefore, the image of $\delta_{K/F}(B)$ in $H^3(F,\mu_2)/(F^\times)\cdot[A]$ is the invariant $\Delta(A)$ defined in \cite[\S 11]{GPT09} (Note that $[B]=\res_{K/F}[A]$, so by the projection formula $\cor_{K/F}((K^\times)\cdot[B])=(N_{K/F}(K^\times))\cdot[A]\subset (F^\times)\cdot[A]$).
\end{remark}
The main property of the invariant is the following result:
\begin{proposition}\label{prop2.2}
The algebra $B$ has a descent to $F$ if and only if $\delta_{K/F}(B)=0$.
\end{proposition}
\begin{proof}
It follows from the definition that $\delta_{K/F}(B)=0$ if and only if there exists $\lambda\in K^\times$ such that the equivalent conditions (i) and (ii) of Lemma \ref{lm4.2} hold. We are going to prove that  $B$ has a descent to $F$ if and only if (i) holds. We know that $B$ has a descent to $F$ if and only if there is $\varphi_0\in I^2F$ with $\dim\varphi_0=6$ such that
\[
B\simeq C_+(\varphi)\simeq C_+(\varphi_0)\otimes K\simeq C_+((\varphi_0)_K).
\]
In other words, $B$ has a descent to $F$ if and only if there exists $\lambda\in K^{\times}$ such that $\varphi\simeq \langle\lambda\rangle\otimes (\varphi_0)_K$ because of the uniqueness of the Albert form up to similarity. Therefore to get the statement, it suffices to show for given $\lambda\in K^\times$, there exists $\varphi_0\in I^2F$ with $\dim\varphi_0=6$ and $\varphi\simeq \langle\lambda\rangle\otimes (\varphi_0)_K$ if and only if  $s_*(\langle\lambda\rangle\otimes\varphi)=0$. 

Suppose $B$ has a descent to $F$, that is,  $\varphi\simeq \langle\lambda\rangle\otimes (\varphi_0)_K$. We have automatically $s_*(\langle\lambda\rangle\otimes\varphi)=0$. Conversely, assume that  $s_*(\langle\lambda\rangle\otimes\varphi)=0$.  Then, as in the proof of Lemma \ref{lm4.1} there exists a quadratic form $\varphi_0\in I^2F$ with $\langle\lambda\rangle\otimes\varphi\simeq (\varphi_0)_K$. That concludes the proof. 
\end{proof}
Now, let us denote by $X$ the Weil transfer of $\SB(B)$ over $F$. Such a transfer exists (see for instance \cite[(2.8)]{BS65} or \cite[Chapter 4]{Sch94}) and is projective (\cite[Corollary 2.4]{Kar00}) since $\SB(B)$ is a projective variety.  Moreover, we have 
\[
X_K\simeq \SB(B)\times \SB(B)
\]
(see \cite[(2.8)]{BS65}). We denote by $F(X)$ the function field of $X$. Notice that $F(X)$ splits $B$; that also means $\varphi$ is hyperbolic over $F(X)$.

For any integer $d\geq 1$, let $\mathbb Q/\mathbb Z(d-1)=\varinjlim \mu_n^{\otimes (d-1)}$, where $\mu_n$ is the group of $n$-th  roots of unity in $F_s$. We let (see \cite[Appendix A]{GMS03})
$H^d(F,\mathbb Q/\mathbb Z(d-1))$ be the Galois cohomology group with coefficients in $\mathbb Q/\mathbb Z(d-1)$. By definition, one has the canonical map $H^3(F,\mu_{2n}^{\otimes 2})\,\rightarrow\, H^3(F,\QZ(2))$. On the other hand, using the surjectivity of the map $H^3(F,\mu_{2n}^{\otimes 2}) \rightarrow H^3(F,\mu_n^{\otimes 2})$ (see \cite{MS82}), the infinite long exact sequence in cohomology induced by the natural exact sequence of Galois modules
\[
1\rightarrow \mu_2 \rightarrow\mu_{2n}^{\otimes 2}\rightarrow \mu_{n}^{\otimes 2}\rightarrow 1
\]
shows that the canonical map $ H^3(F,\mu_2)\,\rightarrow\, H^3(F,\QZ(2))$ is injective.

Since $\varphi$ is hyperbolic over $F(X)$, it is clear that 
\[
e_3(s_*(\varphi))\in \Ker\Big(H^3(F,\QZ(2))\rightarrow H^3(F(X),\QZ(2))\Big).
\]
On the other hand, since $F(X)$ splits $B$, we have
\[
\cor_{K/F}((K^\times)\cdot [B])\subset \Ker\Big(H^3(F,\QZ(2))\rightarrow H^3(F(X),\QZ(2))\Big).
\]
We have just proved the following consequence:
\begin{corollary}
The invariant $\delta_{K/F}(B)$ is in the quotient group
\begin{equation}\label{eq1}
\frac{\Ker\Big(H^3(F,\QZ(2))\rightarrow H^3(F(X),\QZ(2))\Big)}{\cor_{K/F}((K^\times)\cdot [B])}.
\end{equation}
\end{corollary} 
\begin{remark}\label{rmq2.2}
\begin{sloppypar}
The quotient (\ref{eq1}) is canonically identified with the torsion $\CH^2(X)_\tors$  (see \cite{Pey95b} or \cite{Pey95}). Therefore, we may view $\delta_{K/F}(B)$ as belonging to this group.
\end{sloppypar}
\end{remark}
We now study the behavior of $\delta_{K/F}(B)$ under an odd degree extension. Let $\mathbb{F}$ be an odd degree extension of $F$. We denote by $K\mathbb F$ the quadratic extension $\mathbb{F}(\sqrt{a})$. The following result shows that if $B$ has no descent to $F$, then the same holds for $B_{\mathbb F}$.
\begin{proposition}\label{prop2.3}
The scalar extension map
\begin{multline*}
\frac{\Ker\Big(H^3(F,\QZ(2))\rightarrow H^3(F(X),\QZ(2))\Big)}{\cor_{K/F}((K^\times)\cdot [B])}\longrightarrow \\
\frac{\Ker\Big(H^3(\mathbb{F},\QZ(2))\rightarrow H^3(\mathbb{F}(X_{\mathbb{F}}),\QZ(2))\Big)}{\cor_{K\mathbb{F}/\mathbb{F}}(((K\mathbb{F})^{\times})\cdot[B_{K\mathbb F}])}
\end{multline*}
is an injection.
\end{proposition}
\begin{proof}
Let $M\supset K\supset F$ be a separable biquadratic  extension of $K$ in $B$ which splits $B$; so $[M:F]=8$. Since  $M$ splits $B$,  we have
\[
X_M\simeq \SB(B)_M\times \SB(B)_M\simeq {\bf P}^3_M\times {\bf P}^3_M
\]
where ${\bf P}^3_M$ denotes the projective $3$-space considered as variety. The function field $M({\bf P}^3_M)$ of ${\bf P}^3_M$ being a purely transcendental extension of $M$, the function field $M(X_M)$  is a purely transcendental extension of $M$. So, we have
\[
H^3(M,\QZ(2)) \hookrightarrow  H^3(M(X_M),\QZ(2)).
\]
Let $\xi \in \Ker\Big(H^3(F,\QZ(2))\rightarrow H^3(F(X),\QZ(2))\Big)$ and denote by $\xi_{\mathbb{F}}$ its image in $H^3(\mathbb{F},\QZ(2))$ under the scalar extension. Suppose $\xi_{\mathbb{F}}= \cor_{K\mathbb{F}/\mathbb{F}}((\lambda)\cdot[B_{K\mathbb F}])$ for some $\lambda\in (K\mathbb{F})^{\times}$. Let $A$ be a central simple $F$-algebra such that $r_{K/F}(A)=B$. Note that such an algebra $A$ exists because the sequence
\[
\xymatrix @C=3pc{ 
    \text{Br}_2(F) \ar[r]^{r_{K/F}} & \text{Br}_2(K) \ar[r]^{\text{cor}_{K/F}} & \text{Br}_2(F)
     }
\]
is exact.
 One has
\begin{eqnarray*}
[\mathbb{F}:F]\xi= \text{cor}_{\mathbb{F}/F}(\xi_{\mathbb{F}}) & = & \text{cor}_{\mathbb{F}/F}\Big(\cor_{K\mathbb{F}/\mathbb{F}}((\lambda)\cdot[B_{K\mathbb F}])\Big)\\
& = & \text{cor}_{\mathbb{F}/F}\Big(N_{K\mathbb{F}/\mathbb{F}}(\lambda)\cdot [A_\mathbb{F}] \Big)\\
& = & N_{\mathbb{F}/F}\Big(N_{K\mathbb{F}/\mathbb{F}}(\lambda)\Big)\cdot [A]\\
& = & N_{K/F}\Big(N_{K\mathbb{F}/K}(\lambda)\Big)\cdot [A]\\
& = & \cor_{K/F}\Big(N_{K\mathbb{F}/K}(\lambda)\cdot [A_K]\Big) \in  \cor_{K/F}((K^{\times})\cdot [B]).
\end{eqnarray*}
This implies that the order of $\xi$ is odd in  
\[
\frac{\Ker\Big(H^3(F,\QZ(2))\rightarrow H^3(F(X),\QZ(2))\Big)}{\cor_{K/F}((K^{\times})\cdot [B])}
\]
since $[\mathbb{F}:F]$ is odd. On the other hand, consider the following commutative diagram where the vertical maps are given by scalar extension and the horizontal maps are restriction and corestriction maps
\[
\xymatrix @R=3pc @C=2pc { 
    H^3(F,\QZ(2)) \ar[r] \ar[d] & H^3(M,\QZ(2)) \ar[r] \ar@{^{(}->}[d] & H^3(F,\QZ(2)) \ar[d]  \\
     H^3(F(X),\QZ(2)) \ar[r]   & H^3(M(X_M),\QZ(2)) \ar[r]  & H^3(F(X),\QZ(2))
     }
\]
The image of $\xi$ by the top row is $[M:F]\xi=8\xi$. Moreover, a diagram chase shows that $8\xi$ is trivial in $H^3(F, \QZ(2))$ because $\xi$ is in the kernel of the left vertical map, and the central vertical map is injective. So, the order of $\xi$ is then prime to $[\mathbb{F}:F]$. It follows that $\xi\in \cor_{K/F}((K^{\times})\cdot [B])$. The proof is complete.
\end{proof} 

Now, assume that $A$ a central simple algebra of degree $8$ and exponent $2$ over $F$. Let  $K=F(\sqrt{a})$ be a quadratic field extension of $F$ contained in $A$ and let $\mathbb F$ be an odd degree extension of $F$. Denote by $B= C_AK$ the centralizer of $K$ in $A$.
\begin{remark}\label{rmq2.3}
(1) The algebra $A$ admits a decomposition adapted to $K$ if and only if $B$ has a descent to $F$. In other words, $A$ admits a decomposition adapted to $K$ if and only if $\delta_{K/F}(B)=0$. So, $\delta_{K/F}(B)\ne 0$ if $A$ is indecomposable.

\noindent
(2) One deduces from Proposition \ref{prop2.3} that if $K$ is not in a quaternion subalgebra of $A$, then the same holds over $\mathbb F$.

\noindent
(3) Let $\Ker(\text{res})$ be the kernel of the restriction map 
\[
H^3(F,\QZ(2))\longrightarrow H^3(F(\SB(A)),\QZ(2)).
\]
The quotient group
\[
\Ker(\text{res})/[A]\cdot H^1(F,\QZ(2))
\]
is identified with $\CH^2(\SB(A))_\tors$ in \cite{Pey95}. Arguing as in the proof of Proposition \ref{prop2.3}, we may see that the scalar extension map
\[
\CH^2(\SB(A))_\tors \longrightarrow \CH^2(\SB(A)_{\mathbb F})_\tors
\]
is injective. This injection may be also deduced from \cite[Corollary 1.2 and Proposition 1.3]{Kar98}.
\end{remark}

Let $A'$ be the division algebra Brauer equivalent to
\[
A'\sim A\otimes_F (a,t)_{F(t)}
\]
where $t$ is an indeterminate over $F$. Note that 
\[
A'_K\sim A_K\otimes K(t)\sim B_{K(t)}.
\]
The following proposition shows that if $\delta_{K/F}(B)$ is nonzero then the invariant $\Delta(A')$, defined in \cite[\S 11]{GPT09}, is nonzero.
\begin{proposition}\label{prop2.4}
The scalar extension map
\begin{equation}\label{eq4.2}
\frac{H^3(F,\mu_2)}{\cor_{K/F}((K^\times)\cdot[B])}\longrightarrow \frac{H^3(F(t),\mu_2)}{(F(t)^\times)\cdot[A']}
\end{equation}
is an injection.
\end{proposition} 
\begin{proof}
The map is clearly well-defined because
\[
\cor_{K/F}(K^\times\cdot[B])\subset \cor_{K(t)/F(t)}(K(t)^\times\cdot[B_{K(t)}])\subset (F(t)^\times)\cdot[A'].
\]
Let $\xi\in H^3(F,\mu_2)$. Suppose $\xi= f[A']=f([A]+[(a,t)])$ for some $f\in F(t)^\times$. Consider the residue map
\[
\partial_t:\, H^3(F(t), \mu_2)\,\longrightarrow \, H^2(F,\mu_2)
\]
 where $F(t)$ is equipped with its $t$-adic valuation $v_t$ (see for instance \cite[\S 7]{GMS03}). We have 
\[
\partial_t(\xi)=\partial_t(f[A]+f[(a,t)])=0.
\]
Since $f$ is taken modulo $F(t)^{\times 2}$, we may assume that $v_t(f)$ is either $0$ or $1$. 

If $v_t(f)=1$, that is, $f=tf_0$ for some $t$-adic unit $f_0$, then
\begin{eqnarray*}
\partial_t(\xi)=\partial_t(f[A]+f[(a,t)]) & =& [A]+\partial_t((tf_0, a,t))\\
& = & [A]+ \partial_t((-f_0, a,t))\\
& = & [A]+ [(-f_0(0), a)]=0.
\end{eqnarray*}
Therefore $A$ is Brauer equivalent to $(-f_0(0), a)$. It follows that the two quotients of the map (\ref{eq4.2})  are trivial. In this case, there is nothing to show. 

Suppose $v_t(f)=0$. One has
\[
\partial_t(f[A]+f[(a,t)])=[(f(0),a)]=0.
\]
On the other hand, $\xi=f(0)[A]$ by the specialization map $\ker\partial_t \to H^3(F,\mu_2)$ associated with $t$ at $0$. Since $(f(0),a)=0$, we deduce that $\xi\in (N_{K/F}(K^\times))\cdot[A]=\cor_{K/F}((K^\times)\cdot[B])$. The proof is complete.
\end{proof}
\begin{remark}\label{rmq2.4}
According to Remark \ref{rem2.1}, the image of $\delta_{K/F}(B)$ by the scalar extension map (\ref{eq4.2}) is $\Delta(A')$. The proof of Theorem \ref{thmB} provides an example of indecomposable algebra $A'$ of degree $8$ and exponent $2$ such that $\Delta(A')$ is nonzero.
\end{remark}
\section{Proofs of the main statements}
\subsection{Proof of Theorem \ref{thmA}}
We start out by the following lemma (a part of \cite[Theorem 2]{Kah90}):
\begin{lemma}\label{lm3.1}
Every central simple $F$-algebra of exponent $2$ is Brauer equivalent to a tensor product of at most $\frac{1}{2}u(F)-1$ quaternion algebras.
\end{lemma}
\begin{proof}
Let $A$ be a central simple $F$-algebra of exponent $2$. Let $\varphi\in I^2F$ be an anisotropic quadratic form such that $C(\varphi)\sim A$ (by Merkurjev's Theorem \cite{Mer81}). The form $\varphi$ being anisotropic, $\dim\varphi \le u(F)$. Let $\varphi'$ be a subform of $\varphi$ of codimension $1$. The algebra $C(\varphi)$ is Brauer equivalent to $C_0(\varphi')$ which is a tensor product of $\frac{1}{2}(\dim(\varphi')-1)=\frac{1}{2}(\dim\varphi)-1$ quaternion algebras.
\end{proof}
Now, let $A$ be a degree $8$ and exponent $2$ algebra over a field of $u$-invariant smaller than or equal to $8$. By Lemma \ref{lm3.1}, $A$ is Brauer equivalent to a tensor product of three quaternion algebras. Therefore, this equivalence is an isomorphism by dimension count; this proves that $A$ is decomposable. This concludes the proof of Theorem \ref{thmA}.

\subsection{A consequence}
Let $U$ and $V$ be smooth complete geometrically irreducible varieties over $F$. In the appendix, Merkurjev gives conditions under which the scalar extension map $\CH (V)\longrightarrow \CH(V_{F(U)})$ is injective. For instance, let $A$ be central simple $F$-algebra of degree $8$ and exponent $2$ and let $K$ be a quadratic separable extension of $F$ contained in $A$. Denote by $X$ the Weil transfer of $\SB(C_AK)$ over $F$. Let $q$ be a quadratic form over $F$ with $\dim q\geq 9$. It follows from Theorem \ref{thmA.2} that the scalar extension map
\begin{equation}\label{eq2}
\CH^2(X)_\tors \longrightarrow \CH^2(X_{F(q)})_\tors
\end{equation}
is injective, where $F(q)$ is the function field  of the projective quadric defined by $q=0$. This property does not hold anymore if one replaces $X$ by $\SB(A)$. Indeed, Theorem \ref{thmA} implies the following: 
\begin{corollary}\label{cor4.1}
Let $A$ be a central simple algebra of degree $8$ and exponent $2$ over $F$. Assume that $\CH^2(\SB(A))_\tors\ne 0$. Then there exist an extension $F'$ of $F$ and a $9$-dimensional quadratic form $q$ defined over $F_\ell$ such that the scalar extension map \[
\CH^2(\SB(A)_{F'})_\tors \longrightarrow \CH^2(\SB(A)_{F'(q)})_\tors
\]
is not injective.
\end{corollary} 
For the proof we need the following  particular class of fields of $u$-invariant at most $8$: starting with any field $F$ over which there is an anisotropic quadratic form of dimension $8$, one defines a tower of fields
\[
F=F_0\subset F_1\subset\cdots \subset F_{\infty}=\bigcup_iF_i =:{\rm M}_8^u(F) \label{u-invariant}
\]
inductively as follows: if $F_{i-1}$ is already given, the field $F_i$ is the composite of all function fields $F_{i-1}(\varphi)$, where $\varphi$ ranges over (the isometry classes of) all $9$-dimensional forms over $F_{i-1}$. Clearly, any $9$-dimensional form over $\mathbb M$ is isotropic. Therefore $u({\rm M}_8^u(F))\le 8$. Such a construction is due to Merkujev (see \cite{Mer92}). 

\begin{proof}[Proof of Corollary \ref{cor4.1}]
\begin{sloppypar}
Let $\mathbb M={\rm M}_8^u(F)$ be an extension as above; recall that $u(\mathbb M)\le 8$. The algebra $A_{\mathbb M}$ being decomposable by Theorem \ref{thmA}, we have $\CH^2(\SB(A)_{\mathbb M})_{\text{tors}}=0$. Since $\mathbb{M}=\bigcup_i F_i$, there exists $\ell$ such that $\CH^2(\SB(A)_{F_\ell})_{\text{tors}}\ne 0$ and $\CH^2(\SB(A)_{F_{\ell+1}})_\tors=0$. By definition, $F_{\ell+1}$ is the composite of all function fields $F_\ell(q)$, where $q$ ranges over all $9$-dimensional forms over $F_\ell$. Hence, there exists an extension $F'$ of $F_\ell$ and a $9$-dimensional form $q$ over $F_\ell$ such that $\CH^2(\SB(A)_{F'})_\tors\ne 0$ and $\CH^2(\SB(A)_{F'(q)})_\tors=0$ as was to be shown. 
\end{sloppypar} 
\end{proof}
\subsection{Proof of Theorem \ref{thmB}}
Put $\mathbb M ={\rm M}_8^u(F)$ and $B=C_AK$. As in the previous section we denote by $A'$ be the division algebra Brauer equivalent to $A\otimes_F(a,t)_{F(t)}$.
\begin{proof}[First proof]
 Recall that $\delta_{K/F}(C_AK)$ is in $\CH^2(X)_{\text{tors}}$ by Remark \ref{rmq2.2}, where $X$ is the Weil transfer of $\SB(C_AK)$. Since $\delta_{K/F}(C_AK)\ne 0$ (see Remark \ref{rmq2.3}), it follows by  injection (\ref{eq2}) that the extension $\mathbb M(\sqrt{a})$ is not in a quaternion subalgebra of $A_{\mathbb M}$. By  \cite[Proposition 2.10]{Tig87} (or \cite{Bar}) the algebra $A'_{\mathbb M}$ is an indecomposable algebra of degree $8$ and exponent $2$. This concludes the proof.  
\end{proof}
\begin{proof}[Second proof]
Consider the following injections
\[
\frac{H^3(F,\mu_2)}{\cor_{K/F}((K^\times)\cdot [B])}\hookrightarrow \frac{H^3(\mathbb M,\mu_2)}{\cor_{\mathbb MK/\mathbb M}(((\mathbb MK)^\times)\cdot [B_{\mathbb MK}])}\hookrightarrow \frac{H^3(\mathbb M(t),\mu_2)}{(\mathbb M(t)^\times)\cdot [A'_{\mathbb M}]}
\]
where the first is due to Merkurjev (injection (\ref{eq2})) and the second by Proposition \ref{prop2.4}. Since $\delta_{K/F}(B)\ne 0$ and its image by the composite of these above injections is $\Delta(A'_{\mathbb M})$, we have $\Delta(A'_{\mathbb M})\ne 0$. This also means $A'_{\mathbb M}$ is indecomposable by Garibaldi-Parimala-Tignol \cite[\S 11]{GPT09}.
\end{proof}

\subsection{Proof of Theorem \ref{thmC}}
Here, we also need a particular class of fields of  cohomological dimension at most $3$: starting with any field $F$ over which there is an anisotropic $3$-fold Pfister form, one defines a tower of fields
\[
F=F_0\subset F_1\subset\cdots \subset F_{\infty}=\bigcup_iF_i =:{\rm M}_3^{cd}(F)
\]
inductively as follows: the field $F_{2i+1}$ is the maximal odd degree extension of $F_{2i}$; the field $F_{2i+2}$ is the composite of all the function fields  $F_{2i+1}(\pi)$, where $\pi$ ranges over all $4$-fold Pfister forms over $F_{2i+1}$. The arguments used by  Merkurjev in \cite{Mer92}, show that $cd_2({\rm M}_3^{cd}(F))\le 3$. Merkurjev used such a technique for constructing fields of cohomological dimension $2$ over which there exist anisotropic quadratic forms of dimension  $2d$ for an arbitrary integer $d$, i.e, counterexamples to Kaplansky's conjecture in the theory of quadratic forms.

\begin{sloppypar}
Now, let $A$ be a central simple algebra of degree $8$ and exponent $2$ such that $\CH^2(\SB(A))_\tors\ne 0$. Put $\mathbb M={\rm M}_3^{cd}(F)$ and let $\mathbb F$ be an odd degree extension of $F$. The scalar extension map
\[
\CH^2(\SB(A))_\tors \longrightarrow \CH^2(\SB(A)_{\mathbb F})_\tors
\]
\end{sloppypar}
is injective by Remark \ref{rmq2.3}. On the other hand, let $\pi$ be a $4$-fold Pfister form over $\mathbb F$. It follows from Theorem \ref{thmA.1} that the scalar extension map
\begin{equation}\label{eq3}
\CH^2(\SB(A)_{\mathbb F})_\tors \longrightarrow \CH^2(\SB(A)_{{\mathbb F(\pi)}})_\tors
\end{equation}
is injective. We deduce from these two latter injections that $\CH^2(\SB(A)_{\mathbb M})_\tors\ne 0$; and so $A_{\mathbb M}$ is indecomposable. The algebra $A_{\mathbb M}$ being indecomposable, we must have $cd_2(\mathbb M)> 2$ by Theorem \ref{thm2.1}. Hence, $cd_2(\mathbb M)=3$. This completes the proof.

\bigskip
\subsection*{Acknowledgements}
This work is part of my PhD thesis at Universit\'e catholique de Louvain and Universit\'e Paris 13. I would like to thank my thesis supervisors, Anne Qu\'eguiner-Mathieu and Jean-Pierre Tignol, for directing me towards this problem. I would also like to thank Karim Johannes Becher for suggesting  Theorem \ref{thmA} and the idea of the proof of Lemma \ref{lm4.1}. I am particularly grateful to Alexander S. Merkurjev for providing the appendix of the paper.

\vspace{1cm}

\appendix
\section{}
{\markboth{\MakeUppercase{{A. S. Merkurjev}}}{}}
{\markright{\MakeUppercase {O{\scshape n the} C{\scshape how} G{\scshape roup of} C{\scshape ycles of} C{\scshape odimension $2$}}}}
\begin{center}
{\large O{\scshape n the} C{\scshape how} G{\scshape roup of} C{\scshape ycles of} C{\scshape odimension $2$}}
\end{center}
\begin{center}
\author{{\scshape by} A{\scshape lexander} S. M{\scshape erkurjev}}
\end{center}
\bigskip
Let $X$ be an algebraic variety over $F$. We write $A^i(X,K_n)$ for the homology
group of the complex
\[
\coprod_{x\in X^{(i-1)}}K_{n-i+1}\(F(x)\)\stackrel{\partial}{\llg}
\coprod_{x\in X^{(i)}}K_{n-i}\(F(x)\)\stackrel{\partial}{\llg}
\coprod_{x\in X^{(i+1)}}K_{n-i-1}\(F(x)\),
\]
where $K_j$ are the Milnor $K$-groups and $X^{(i)}$ is the set of points in $X$ of codimension $i$ 
(see \cite[\S 5]{Rost98a}). In particular, $A^i(X,K_i)=\CH^i(X)$ is the Chow group
of classes of codimension $i$ algebraic cycles on $X$.

Let $X$ and $Y$ be smooth complete geometrically irreducible varieties over $F$.

\begin{proposition}\label{prop1}
Suppose that for every field extension $K/F$ we have:
\begin{enumerate}
  \item The natural map $\CH^1(X)\llg \CH^1(X_K)$ is an isomorphism of torsion free groups,
  \item The product map $\CH^1(X_K)\tens K^\times\llg A^1(X_K,K_2)$
is an isomorphism.
\end{enumerate}

\noindent Then the natural sequence
\[
0\llg \(\CH^1(X)\tens \CH^1(Y)\)\oplus\CH^2(Y)\llg\CH^2(X\times Y)\llg \CH^2(X_{F(Y)})
\]
is exact.
\end{proposition}

\begin{proof}
Consider the spectral sequence
 \[
 E_1^{p,q}=\coprod_{y\in Y^{(p)}} A^q(X_{F(y)},K_{2-p})\Longrightarrow A^{p+q}(X\times Y,K_2)
 \]
for the projection $X\times Y\llg Y$ (see \cite[Cor. 8.2]{Rost98a}). The nonzero terms of the first page are the following:

\[
\xymatrix{
\CH^2(X_{F(Y)})  \\
A^1(X_{F(Y)},K_{2}) \ar[r] & \coprod_{y\in Y^{(1)}} \CH^1(X_{F(y)})\\
A^0(X_{F(Y)},K_{2}) \ar[r]& \coprod_{y\in Y^{(1)}} A^0(X_{F(y)},K_{1}) \ar[r]& \coprod_{y\in Y^{(2)}} \CH^0(X_{F(y)}).
}
\]

\medskip

Then $E_1^{2,0}=\coprod_{y\in Y^{(2)}} \Z$
is the group of cycles on $X$ of codimension $2$ and $E_1^{1,0}=\coprod_{y\in Y^{(1)}} F(y)^\times$
as $X$ is complete. It follows that $E_2^{2,0}=\CH^2(Y)$.

By assumption, the differential $E_1^{0,1}\llg E_1^{1,1}$ is identified with the map
\[
\CH^1(X)\tens \Big(F(Y)^\times\llg \coprod_{y\in Y^{(1)}} \Z\Big).
\]
Since $Y$ is complete and $\CH^1(X)$ is torsion free, we have $E_2^{0,1}=\CH^1(X)\tens F^\times$ and $E_\infty^{1,1}=E_2^{1,1}=\CH^1(X)\tens \CH^1(Y)$.

The edge map
\[
A^1(X\times Y,K_2)\llg E_2^{0,1}=\CH^1(X)\tens F^\times
\]
is split by the product map
\[
\CH^1(X)\tens F^\times =A^1(X,K_1)\tens A^0(Y,K_1)\llg A^1(X\times Y,K_2),
\]
hence the edge map is surjective. Therefore, the differential $E_2^{0,1}\llg E_2^{2,0}$ is trivial and hence
$E_\infty^{2,0}= E_2^{2,0}=\CH^2(Y)$. Thus, the natural homomorphism
\[
\CH^2(Y)\llg\Ker\(\CH^2(X\times Y)\llg \CH^2(X_{F(Y)})\)
\]
is injective and its cokernel is isomorphic to $\CH^1(X)\tens \CH^1(Y)$. The statement follows.
\end{proof}

\begin{example}\label{ex1}
Let $X$ be a projective homogeneous variety of a semisimple algebraic group over $F$. 
There exist an \'etale $F$-algebra $E$ and an Azumaya $E$-algebra $A$ such that for $i=0$ and $1$, we have an exact sequence
\[
0\llg A^1(X,K_{i+1})\llg K_i(E) \xra{\rho} H^{i+2}\(F,\Q/\Z(i+1)\),
\]
where $\rho(x)=N_{E/F}\((x)\cup[A]\)$ (see \cite{Merkurjev95} and \cite{MT95}). If the algebras $E$ and $A$ are split, then $\rho$ is trivial 
and for every field extension $K/F$,
\[
\CH^1(X)\simeq K_0(E)\simeq K_0(E\tens K)\simeq \CH^1(X_K), 
\]
\[A^1(X_K,K_{2})\simeq K_1(E\tens K)\simeq K_0(E)\tens K^\times\simeq \CH^1(X_K)\tens K^\times.
\]
Therefore, the condition $(1)$ and $(2)$ in Proposition \ref{prop1} hold. For example, if $X$ is a smooth projective quadric of dimension at least $3$,
then $E=F$ and $A$ is split.
\end{example}

Now consider the natural complex
\begin{equation}\label{seq2}
\CH^2(X)\oplus\(\CH^1(X)\tens\CH^1(Y)\)\llg \CH^2(X\times Y)\llg \CH^2(Y_{F(X)}).
\end{equation}

\begin{proposition}\label{prop2}
Suppose that 
\begin{enumerate}
  \item The Grothendieck group $K_0(Y)$ is torsion-free,
  \item The product map $K_0(X)\tens K_0(Y)\llg K_0(X\times Y)$ is an isomorphism.
\end{enumerate}

\noindent Then the sequence (\ref{seq2}) is exact.
\end{proposition}

\begin{proof}
It follows from the assumptions that the map $K_0(Y)\llg K_0(Y_{F(X)})$ is injective and the kernel of the natural homomorphism
$K_0(X\times Y)\llg K_0(Y_{F(X)})$ coincides with
\[
I_0(X)\tens K_0(Y),
\]
where $I_0(X)$ is  the kernel of the rank homomorphism $K_0(X)\llg \Z$.

The kernel of the second homomorphism in the sequence (\ref{seq2}) is generated by the classes of closed integral subschemes
$Z\subset X\times Y$ that are not dominant over $X$. By Riemann-Roch (see \cite{Grothendieck58}), we have $[Z]=-c_2\([O_Z]\)$ in
$\CH^2(X\times Y)$, where $c_i:K_0(X\times Y)\llg\CH^i(X\times Y)$ is the $i$-th Chern class map. As
\[
[O_Z]\in \Ker(K_0(X\times Y)\llg K_0(Y_{F(X)})\)=I_0(X)\tens K_0(Y),
\]
it suffices to to show that
$c_2\(I_0(X)\tens K_0(Y)\)$ is contained in the image $M$ of the first map in the sequence (\ref{seq2}).

The formula $c_2(x+y)=c_2(x)+c_1(x)c_1(y)+c_2(y)$ shows that it suffices to prove that for all $a,a'\in I_0(X)$ and $b,b'\in K_0(Y)$,
the elements $c_1(ab)\cdot c_1\(a'b')$ and $c_2(ab)$ are contained in $M$.
This follows from the formulas (see \cite[Remark 3.2.3 and Example 14.5.2]{Fulton84}): $c_1(ab)=mc_1(a)+nc_1(b)$ and
\[
c_2(ab)=\frac{m^2-m}{2}c_1(a)^2+mc_2(a)+(nm-1)c_1(a)c_1(b)+\frac{n^2-n}{2}c_1(b)^2+nc_2(b),
\]
where $n=\operatorname{rank}(a)$ and $m=\operatorname{rank}(b)$.
\end{proof}

\begin{example}\label{ex2}
If $Y$ is a projective homogeneous variety, then the condition $(1)$ holds by \cite{Panin94}. 
If $X$ is a projective homogeneous variety of a semisimple algebraic group $G$ over $F$ and the Tits algebras of $G$ are split, then it follows from \cite{Panin94} that
the condition $(2)$ also holds for any $Y$. For example, if the even Clifford algebra of a nondegenerate quadratic form is split, then the corresponding
projective quadric $X$ satisfies $(2)$ for any $Y$.
\end{example}

For any field extension $K/F$, let $K^s$ denote the subfield of elements that are algebraic and separable over $F$.

\begin{proposition}\label{prop3}
Suppose that for every field extension $K/F$ we have:
\begin{enumerate}
  \item The natural map $\CH^1(X)\llg \CH^1(X_K)$ is an isomorphism,
  \item The natural map $\CH^1(Y_{K^s})\to \CH^1(Y_K)$ is an isomorphism.
\end{enumerate}

\noindent Then the sequence (\ref{seq2}) is exact.
\end{proposition}

\begin{proof}
Consider the spectral sequence
\begin{equation}\label{ss2}
E_1^{p,q}(F)=\coprod_{x\in X^{(p)}} A^q(Y_{F(x)},K_{2-p})\Longrightarrow A^{p+q}(X\times Y,K_2)
\end{equation}
for the projection $X\times Y\llg X$. The nonzero terms of the first page are the following:

\[
\xymatrix{
\CH^2(Y_{F(X)})  \\
A^1(Y_{F(X)},K_{2}) \ar[r] & \coprod_{x\in X^{(1)}} \CH^1(Y_{F(x)})\\
A^0(Y_{F(X)},K_{2}) \ar[r]& \coprod_{x\in X^{(1)}} A^0(Y_{F(x)},K_{1}) \ar[r]& \coprod_{x\in X^{(2)}} \CH^0(Y_{F(x)}).
}
\]

\medskip

As in the proof of Proposition \ref{prop1}, we have $E_2^{2,0}(F)=\CH^2(X)$. For a field extension $K/F$, write $C(K)$ for the factor group
\[
\Ker\(\CH^2(X_K\times Y_K)\llg \CH^2(Y_{K(X)})\)/\Im\(\CH^2(X_K)\llg \CH^2(X_K\times Y_K)\).
\]
The spectral sequence (\ref{ss2}) for the varieties $X_K$ and $Y_K$ over $K$ yields an isomorphism $C(K)\simeq E_2^{1,1}(K)$.
We have a natural composition
\[
\CH^1(X_K)\tens\CH^1(Y_K)\llg E_1^{1,1}(K)\llg E_2^{1,1}(K)\simeq C(K).
\]

We claim that the group $C(F)$ is generated by images of the compositions 
\[
\CH^1(X_K)\tens\CH^1(Y_K)\llg  C(K)\xra{N_{K/F}} C(F)
\]
over all finite separable field extensions $K/F$ (here $N_{K/F}$ is the norm map for the extension $K/F$).

The group $C(F)$ is generated by images of the maps
\[
\varphi_x:\CH^1(Y_{F(x)})\llg E_2^{1,1}(F)\simeq C(F)
\]
over all points $x\in X^{(1)}$.
Pick such a point $x$ and let $K:=F(x)^s$ be the subfield of elements that are separable over $F$. Then $K/F$ is a finite separable field extension. Let $x'\in X_K^{(1)}$
be a point over $x$ such that $K(x')\simeq F(x)$. Then $\varphi_x$ coincides with the composition
\[
\CH^1(Y_{K(x')})\llg  C(K)\xra{N_{K/F}} C(F).
\]
By assumption, the map $\CH^1(Y_{K})\llg \CH^1(Y_{K(x')})$ is an isomorphism, hence the image of $\varphi_x$ coincides with the image of
\[
[x']\tens\CH^1(Y_K)\llg  C(K)\xra{N_{K/F}} C(F),
\]
whence the claim.

As $\CH^1(X)\llg \CH^1(X_K)$ is an isomorphism for every field extension $K/F$, the projection formula shows  that the map
$\CH^1(X)\tens\CH^1(Y)\llg   C(F)$ is surjective. The statement follows.
\end{proof}

\begin{example}\label{ex3}
Let $Y$ be a projective homogeneous variety with the $F$-algebras $E$ and $A$ as in Example \ref{ex1}. If $A$ is split, then $\CH^1(Y_K)=K_0(E\tens K)$
for every field extension $K/F$. As $K^s$ is separably closed in $K$, the natural map $K_0(E\tens K^s)\llg K_0(E\tens K)$ is an isomorphism, therefore, the
condition $(2)$ holds.
\end{example}

Write $\widetilde{\CH}^2(X\times Y)$ for the cokernel of the product map $\CH^1(X)\tens \CH^1(Y)\llg \CH^2(X\times Y)$. We have  the following commutative diagram:

\[
\xymatrix{
& & \CH^2(X) \ar[d] \ar[rd] \\
0 \ar[r] &  \CH^2(Y) \ar[r]\ar[rd] & \widetilde{\CH}^2(X\times Y)\ar[r]\ar[d] & \CH^2(X_{F(Y)})\\
& & \CH^2(Y_{F(X)})
}
\]

\medskip

Proposition \ref{prop1} gives conditions for the exactness of the row in the diagram and Propositions \ref{prop2} and \ref{prop3} - for the exactness of the column in the diagram.

A diagram chase yields together with Propositions \ref{prop1}, \ref{prop2} and \ref{prop3} yields the following statements.

\begin{theorem}\label{thmA.1}
Let $X$ and $Y$ be smooth complete geometrically irreducible varieties such that for every field extension $K/F$:
\begin{enumerate}
  \item The natural map $\CH^1(X)\llg \CH^1(X_K)$ is an isomorphism of torsion free groups,
   \item The natural map $\CH^2(X)\llg \CH^2(X_K)$ is injective,
  \item The product map $\CH^1(X_K)\tens K^\times\llg A^1(X_K,K_2)$
is an isomorphism,
  \item The Grothendieck group $K_0(Y)$ is torsion-free,
  \item The product map $K_0(X)\tens K_0(Y)\llg K_0(X\times Y)$ is an isomorphism.
 \end{enumerate}

\noindent Then the natural map $\CH^2(Y)\llg \CH^2(Y_{F(X)})$ is injective.
\end{theorem}

\begin{remark}\label{rem1}
The conditions $(1)-(3)$ hold for  a smooth projective quadric $X$ of dimension at least $7$ by \cite[Theorem 6.1]{Karpenko90} and Example \ref{ex1}.
By Example \ref{ex2}, the conditions $(4)$ and $(5)$ hold if the even Clifford algebra of $X$ is split and $Y$ is a projective homogeneous variety.
\end{remark}

\begin{theorem}\label{thmA.2}
Let $X$ and $Y$ be smooth complete geometrically irreducible varieties such that for every field extension $K/F$:
\begin{enumerate}
  \item The natural map $\CH^1(X)\llg \CH^1(X_K)$ is an isomorphism of torsion free groups,
   \item The natural map $\CH^2(X)\llg \CH^2(X_K)$ is injective,
  \item The product map $\CH^1(X_K)\tens K^\times\llg A^1(X_K,K_2)$
is an isomorphism,
  \item The natural homomorphism $\CH^1(Y_{K^s})\to \CH^1(Y_K)$ is an isomorphism.
 \end{enumerate}

\noindent Then the natural map $\CH^2(Y)\llg \CH^2(Y_{F(X)})$ is injective.
\end{theorem}

\begin{remark}
The conditions $(1)-(3)$ hold for  a smooth projective quadric $X$ of dimension at least $7$ and a projective homogeneous variety $Y$ with the split
Azumaya algebra by Remark \ref{rem1} and Example \ref{ex3}.
\end{remark}

\end{document}